\documentclass[12pt, reqno]{amsart}
\usepackage{amsmath, amsthm, amscd, amsfonts, amssymb, graphicx, color}
\usepackage[bookmarksnumbered,plainpages]{hyperref}

\newtheorem{theorem}{Theorem}[section]

\newtheorem{proposition}[theorem]{Proposition}
\newtheorem{corollary}[theorem]{Corollary}
\theoremstyle{definition}

\theoremstyle{remark}
\newtheorem{remark}[theorem]{Remark}
\numberwithin{equation}{section}

\newcommand{\A}{{\mathfrak A}}
\newcommand{\B}{{\mathfrak B}}

\newcommand{\M}{\mathcal{M}}
\newcommand{\N}{\mathcal{N}}
\newcommand{\X}{\mathcal{X}}
\newcommand{\Ha}{\mathcal{H}}

\newcommand{\ip}[2]{\langle#1,#2\rangle}
\newcommand{\abs}[1]{\lvert#1\rvert}

\begin{document}

\title[Perturbation of the Wigner equation]{Perturbation of the
Wigner equation in inner product $C^*$-modules}

\author[J.\ Chmieli\'nski]{Jacek Chmieli\'nski}
\address{Institute of Mathematics,
Pedagogical University of Cracow\\
Pod\-cho\-r\c{a}\-\.{z}ych 2, 30-084 Krak\'{o}w, Poland}

\email{jacek@ap.krakow.pl}

\author[D.\ Ili\v{s}evi\'c]{Dijana Ili\v{s}evi\'c}
\address{Department of Mathematics, University of Zagreb\\
Bijeni\v{c}ka 30, P.O. Box 335, 10 002 Zagreb, Croatia}

\email{ilisevic@math.hr}

\author[M.S.\ Moslehian]{Mohammad~Sal~Moslehian}
\address{%
Department of Mathematics, Ferdowsi University of Mashhad\\
P. O. Box 1159, Mashhad 91775, Iran; \newline Centre of Excellence
in Analysis on Algebraic Structures (CEAAS), Ferdowsi University of Mashhad, Iran.\\
} \email{moslehian@ferdowsi.um.ac.ir}

\author[Gh.\ Sadeghi]{Ghadir Sadeghi}
\address{%
Department of Mathematics, Ferdowsi University of Mashhad\\
P. O. Box 1159, Mashhad 91775, Iran\\
Banach Mathematical Research Group (BMRG), Mashhad, Iran.}

\email{ghadir54@yahoo.com}



\date{}

\begin{abstract}
Let $\A$ be a $C^*$-algebra and $\B$ be a von Neumann algebra that
both act on a Hilbert space $\Ha$. Let $\M$ and $\N$ be inner
product modules over $\A$ and $\B$, respectively. Under certain
assumptions we show that for each mapping $f\colon{\mathcal M} \to
{\mathcal N}$ satisfying
$$\|\,|\ip{f(x)}{f(y)}|-|\ip{x}{y}|\,\|\leq\varphi(x,y)\qquad (x,y\in{\mathcal M}),$$ where $\varphi$ is
a control function, there exists a solution $I\colon{\mathcal M} \to
{\mathcal N}$ of the Wigner equation
$$|\ip{I(x)}{I(y)}|=|\ip{x}{y}|\qquad (x, y \in {\mathcal M})$$
such that
$$\|f(x)-I(x)\|\leq\sqrt{\varphi(x,x)} \qquad (x\in {\mathcal M}).$$
\end{abstract}

\maketitle

\section{Introduction and preliminaries}

In this paper, we deal with a perturbation of the Wigner equation,
establishing a link between two topics: inner product $C^*$-modules
and stability of functional equations.


\subsection{Inner product $C^*$-modules}

A \emph{$C^*$-algebra} is a Banach $*$-algebra $(\A, \Vert \cdot
\Vert)$ such that $\Vert a^*a \Vert = \Vert a \Vert^2$ for every $a
\in \A$. Every $C^*$-algebra can be regarded as a $C^*$-subalgebra
of $\mathbb{B}\,(\Ha)$, the algebra of all bounded linear operators
on some Hilbert space $\Ha$. Recall that $a \in \A$ is called
\emph{positive} (we write $a \geq 0$) if $a=b^*b$ for some $b \in
\A$. If $a \in \A$ is positive, then there is a unique positive $b
\in \A$ such that $a=b^2$; such an element $b$ is called the
positive square root of $a$. For every $a \in \A$, the positive
square root of $a^*a$ is denoted by $\abs{a}$.

A \emph{von Neumann algebra} is a $C^*$-subalgebra of
$\mathbb{B}\,(\Ha)$ which contains the identity operator $id_{\Ha}$
and is closed in the weak operator topology (the topology generated
by the seminorms $a \mapsto |\langle a\xi, \eta\rangle|,$ where
$\xi, \eta \in \Ha$). If $\A$ is a von Neumann algebra, then the
following \emph{polar decomposition} holds: for each $a \in \A$
there exists a partial isometry $u \in \A$ (i.e., $u^*u$ is a
projection) such that $a=u\abs{a}$ and $u^*a=\abs{a}$ (see e.g.
\cite[Theorem 4.1.10]{mur} or \cite[Theorem I.8.1]{dav}).

Let $(\A,\|\cdot\|)$ be a $C^*$-algebra and let $\X$ be an algebraic
right $\A$-module which is a complex linear space with $(\lambda
x)a=x(\lambda a)=\lambda (xa)$ for all $x\in {\mathcal X}$, $a\in
\A$, $\lambda\in {\mathbb C}$. The space ${\mathcal X}$ is called a
\emph{(right) inner product} $\A$-\emph{module} (\emph{inner product
$C^*$-module over the $C^*$-algebra $\A$, pre-Hilbert $\A$-module})
if there exists an $\A$-valued inner product, i.e., a mapping
$\ip{\cdot}{\cdot} \colon {\mathcal X}\times {\mathcal X}\to \A$
satisfying
\begin{enumerate}
\item[{\rm (i)}] $\ip{x}{x}\geq 0$ and $\ip{x}{x} =0$ if and only
if $x=0$,

\item[{\rm (ii)}] $\ip{x}{\lambda y+z} = \lambda\ip{x}{y}+
\ip{x}{z}$,

\item[{\rm (iii)}] $\ip{x}{ya} =\ip{x}{y}a$,

\item[{\rm (iv)}] $\ip{y}{x}=\ip{x}{y}^*$,
\end{enumerate}
for all $x, y, z \in {\mathcal X}$, $a\in \A$, $\lambda\in {\mathbb
C}$. The conditions (ii) and (iv) yield the fact that the inner
product is conjugate-linear with respect to the first variable.
Elements $x,y \in \X$ are called orthogonal if and only if $\langle
x, y \rangle = 0$. In an inner product $\A$-module $\X$ the
following version of the \emph{Cauchy-Schwarz inequality} is true:
$$\Vert \langle x, y \rangle \Vert \leq \Vert x \Vert_{\X} \Vert y \Vert_{\X} \qquad (x,y \in \X),$$
where $\Vert x \Vert_{\X} = \sqrt{\Vert \langle x, x \rangle \Vert}$
for all $x \in \X$ (see e.g.~\cite[Proposition 1.2.4.(iii)]{M-T}).
It follows that $\Vert \cdot \Vert_{\X}$ is a norm on $\X$, so $(\X,
\Vert \cdot \Vert_{\X})$ is a normed space (we will denote the norm
in $\X$ simply by $\|\cdot\|$, omitting the subscript $\X$; it will
be clear from the context whether $\Vert \cdot \Vert$ denotes the
norm on $\A$ or the norm on $\mathcal{X}$). If this normed space is
complete, then $\X$ is called a \emph{Hilbert} $\A$-\emph{module},
or a \emph{Hilbert $C^*$-module over the $C^*$-algebra} $\A$. A left
inner product $\A$-module can be defined analogously. Any inner
product (resp.~Hilbert) space is an inner product (resp.~Hilbert)
$\mathbb{C}$-module and any $C^*$-algebra $\A$ is a Hilbert
$C^*$-module over itself via $\langle a, b \rangle =a^*b$, for all
$a,b \in \A$.

For an inner product $\A$-module ${\mathcal X}$, let ${\mathcal
X}^\#$ be the set of all bounded $\A$-linear mappings from
${\mathcal X}$ into $\A$, that is, the set of all bounded linear
mappings $f : \X \to \A$ such that $f(xa)=f(x)a$ for all $x \in \X$,
$a \in \A$. Every $x \in {\mathcal X}$ gives rise to a mapping
$\hat{x} \in {\mathcal X}^\#$ defined by $\hat{x}(y) =\ip{x}{y}$ for
all $y \in {\mathcal X}$. A Hilbert module ${\mathcal X}$ is called
\emph{self-dual} if ${\mathcal X}^\# = \{\hat{x}: x\in {\mathcal
X}\}$.

More information on inner product modules can be found e.g.~in
monographs \cite{LAN, M-T}.


\subsection{Stability of functional equations}
Defining, in some way, the class of approximate solutions of the
given functional equation, one can ask whether each mapping from
this class can be somehow approximated by an exact solution of the
considered equation. Such a problem was formulated by Ulam in 1940
(cf.\ \cite{ulam}) and solved in the next year for the Cauchy
functional equation by Hyers \cite{hyers}. It gave rise to the
\emph{stability theory} for functional equations. Subsequently,
various approaches to the problem have been introduced by several
authors. For the history and various aspects of this theory we refer
the reader to monographs \cite{hir, jung-book}. Recently, the stability problems have been investigated in Hilbert
$C^*$-modules as well; see \cite{amy, a-m}.


\subsection{Wigner equation}

We will be considering the \emph{Wigner equation}
$$
|\ip{I(x)}{I(y)}|=|\ip{x}{y}|\qquad (x, y \in {\mathcal M}),
\eqno{{\rm (W)}}
$$
where $I\colon {\mathcal M} \to {\mathcal N}$ is a mapping between
inner product  modules ${\mathcal M}$ and ${\mathcal N}$ over
certain $C^*$-algebras.

We say that two mappings $f,g\colon{\mathcal M}\to{\mathcal N}$
are \emph{phase-equivalent} if and only if there exists a scalar
valued mapping $\xi\colon{\mathcal M}\to{\mathbb C}$ such that
$|\xi(x)|=1$ and $f(x)=\xi(x)g(x)$ for all $x\in{\mathcal M}$. The
equation (W) has been already introduced in 1931 by E.P.\ Wigner
\cite{wigner} in the realm of (complex) Hilbert spaces. The
classical Wigner's theorem, stating that a solution of (W) has to
be phase-equivalent to a unitary or antiunitary operator, has deep
applications in physics, see \cite{pos, sch}. One of the proofs of
this remarkable result can be found e.g.\ in \cite{gyory} (for
further comments we refer also to \cite{raetz}). Recently,
Wigner's result has been studied in the realm of Hilbert modules
(cf.~e.g.~\cite{b-g1, b-g2, molnar}). The stability of the Wigner
equation has been extensively studied for Hilbert spaces only
(cf.\ a survey paper \cite{jch-survey} or \cite[Chapter 9]{hir}).

In the following section we consider the stability of the Wigner
equation in the setting of inner product  modules. Let us mention
that the stability of the related \emph{orthogonality equation}
$$
\ip{I(x)}{I(y)}=\ip{x}{y}
$$
in this framework has been recently established in \cite{ch-m}.


\section{Stability of the Wigner equation}

Suppose that we are given a {\it control mapping} $\varphi\colon
{\mathcal M}\times{\mathcal M}\to[0,\infty)$ satisfying, with some
constant $0<c\neq 1$, the following pointwise convergence and
boundedness:

\begin{eqnarray}\label{phi}
\begin{array}{ll}
\mbox{(a)} & \displaystyle{\lim_{n\to\infty}} c^{n}\varphi(c^{-n}x, y)=0\ \mbox{and\ }
\displaystyle{\lim_{n\to\infty}} c^{n}\varphi(x, c^{-n}y)=0\ \mbox{for any fixed\ } x,y \in {\mathcal M};\\
\\
\mbox{(b)} & \mbox{the sequence } \big(c^{2n}\varphi(c^{-n}x,
c^{-n}x)\big)\ \mbox{is bounded for any fixed\ } x \in
\mathcal{M}.
\end{array}
\end{eqnarray}

We say that a mapping $f\colon{\mathcal M} \to {\mathcal N}$ {\it
approximately} satisfies the Wigner equation if
$$
\|\,|\ip{f(x)}{f(y)}|-|\ip{x}{y}|\,\|\leq\varphi(x,y)\qquad
(x,y\in{\mathcal M}). \eqno{\mbox{(W$_{\varphi}$)}}
$$
The question we would like to answer is if each solution of
(W$_{\varphi}$) can be approximated by a solution of (W).

Let us consider the following condition on an inner product
$C^*$-module ${\mathcal X}$.


\begin{enumerate}
\item[{\sf [H]}] For each norm-bounded sequence $(x_n)$ in
${\mathcal X}$, there exists a subsequence $(x_{l_n})$ of $(x_n)$
and $x_0\in{\mathcal X}$ such that
\begin{equation*}\label{weaklim}
\|\ip{x_{l_n}}{y}- \ip{x_0}{y}\| \to 0\ (\mbox{as}\
n\to\infty)\qquad \mbox{for all}\ y\in {\mathcal X}.
\end{equation*}
\end{enumerate}

Validity of {\sf [H]} in Hilbert spaces follows from its
reflexivity and the fact that each ball is sequentially weakly
compact. It is an interesting question to characterize the class
of all inner product $C^*$-modules in which {\sf [H]} is
satisfied. In \cite{frank} the following similar condition is considered.

\begin{enumerate}
\item[{\sf [F]}] The unit ball of ${\mathcal X}$ is complete with
respect to the topology which is induced by the semi-norms
$x\mapsto \|\ip{x}{y}\|$ with $y\in {\mathcal X}$, $\|y\|\leq 1$.
\end{enumerate}

We have the following result.


\begin{proposition}\label{prop}
If ${\mathcal X}$ is a Hilbert $C^*$-module over a
finite-dimensional $C^*$-algebra $\A$, then the condition {\sf [H]}
is satisfied.
\end{proposition}
\begin{proof}
Let $(x_n)$ be a bounded sequence in ${\mathcal X}$ ($\|x_n\|\leq
M$, $n=1,2,\ldots$). Then the $\A$-valued sequence $(\ip{x_n}{x_1})$
is bounded ($\|\ip{x_n}{x_1}\|\leq M\,\|x_1\|$). Since $\A$ is
finite-dimensional, the theorem of Bolzano-Weierstrass holds true,
so there exists a subsequence $(\ip{x_n^1}{x_1})$ of the sequence
$(\ip{x_n}{x_1})$ convergent in $\A$. Next, we may choose, by the
same reason, a convergent subsequence $(\ip{x_n^2}{x_2})$ of the
bounded sequence $(\ip{x_n^1}{x_2})$, and so on. Define
$x_{l_n}:=x_{n}^{n}$. Obviously, $\left(x_{l_n}\right)$ is a
subsequence of $(x_n)$. It is also clear from the construction of
$\left(x_{l_n}\right)$ that $\ip{x_{l_n}}{x_i}=\ip{x_{n}^{n}}{x_i}$
is  convergent in $\A$ for $i=1,2,\ldots$ (when $n\to\infty$).
Therefore also $\ip{x_{l_n}}{z}$ is convergent for all $z$ in
${\mathcal X}_{0}$
--- the closed submodule of ${\mathcal X}$ generated by the sequence
$\{x_1,x_2,\ldots\}$. Since $\A$ is finite-dimensional, the Hilbert
$\A$-module ${\mathcal X_0}$ is self-dual (cf.\ \cite[p.~27]{M-T})
and we have  ${\mathcal X}={\mathcal X}_{0}\oplus {\mathcal
X}_{0}^{\bot}$ (cf. \cite[Proposition 2.5.4]{M-T}). For any
$y\in{\mathcal X}$, $y=z+z'$ with some $z\in{\mathcal X}_{0}$ and
$z'\in  {\mathcal X}_{0}^{\bot}$. Thus
$\ip{x_{l_n}}{y}=\ip{x_{l_n}}{z}+\ip{x_{l_n}}{z'}=\ip{x_{l_n}}{z}$,
whence $\ip{x_{l_n}}{y}$ is convergent to some $\varphi(y)$ in $\A$
for any $y\in {\mathcal X}$. The mapping $\varphi$ is $\A$-linear.
Moreover, we have $\Vert \langle x_{l_n}, y \rangle \Vert \leq \Vert
x_{l_n} \Vert \Vert y \Vert \leq M \Vert y \Vert$, whence $\Vert
\varphi(y) \Vert \leq M \Vert y \Vert$ and thus $\varphi\in{\mathcal
X}^{\#}$. From the self-duality of ${\mathcal X}$ there exists
$x_0\in {\mathcal X}$ such that $\varphi(y)=\ip{x_0}{y}$ for all
$y\in{\mathcal X}$. From the definition of $\varphi$ this means
$\ip{x_{l_n}}{y}\to\ip{x_0}{y}$ for all $y\in {\mathcal X}$.
\end{proof}


\begin{theorem}\label{t1}
Let $\A$ be a $C^*$-algebra and $\B$ be a von Neumann algebra that both act on a Hilbert space $\Ha$. Let ${\M}$ be an inner product
$\A$-module and let $\N$ be an inner product $\B$-module satisfying
{\sf [H]}. Then, for each mapping $f\colon{\mathcal M} \to {\mathcal
N}$ satisfying {\rm (W$_{\varphi}$)}, with $\varphi$ satisfying
(\ref{phi}), there exists $I \colon{\mathcal M} \to {\mathcal N}$
with the following properties:
\begin{itemize}
\item[(i)] $\langle I(x), I(x) \rangle = \langle x, x \rangle
\qquad (x \in {\mathcal M})$, \item[(ii)] $I$ preserves
orthogonality in both directions, that is, $\langle x, y \rangle =
0$ if and only if $\langle I(x), I(y) \rangle = 0$,
\item[(iii)]$\|f(x)-I(x)\|\leq\sqrt{\varphi(x,x)} \qquad (x \in
{\mathcal M})$.
\end{itemize}
Furthermore, there exists $h \colon \M \to \N$ such that the
following decomposition holds:
$$f(x)=h(x)+I(x), \quad \langle h(x), I(x) \rangle = 0 \quad \textup{and} \quad
\Vert h(x) \Vert \leq\sqrt{\varphi(x,x)} \qquad (x\in {\mathcal
M}).$$ If $\B$ is abelian, then $I$ can be chosen as a solution of
{\rm (W)}.
\end{theorem}

\begin{proof} For $n \in \mathbb{N} \cup \{0\}$ let $f_n(x):=c^{n}f(c^{-n}x)$,  $x\in
{\mathcal M}$. Substituting in (W$_{\varphi}$), $c^{-n}x$ and
$c^{-m}y$ (with $m,n\in\mathbb{N}\cup\{0\}$) for $x$ and $y$,
respectively, one obtains
\begin{equation}\label{e1}
\|\,|\ip{f_n(x)}{f_m(y)}|-|\ip{x}{y}|\,\|\leq
c^{m+n}\varphi(c^{-n}x,c^{-m}y)\qquad (x, y \in {\mathcal M}).
\end{equation}
By (\ref{phi}-a),
$$\|\,|\ip{f_n(x)}{f(y)}|-|\ip{x}{y}|\,\|\to 0\ (\mbox{as}\ n\to\infty) \qquad (x,y \in \M),$$
which, by the continuity of multiplication, implies
\begin{equation}\label{zero}
\|\,|\ip{f_n(x)}{f(y)}|^2-|\ip{x}{y}|^2\,\|\to 0\ (\mbox{as}\
n\to\infty) \qquad (x,y \in \M).
\end{equation}
Analogously,
$$\|\,|\ip{f(x)}{f_n(y)}|-|\ip{x}{y}|\,\|\to 0\ (\mbox{as}\ n\to\infty) \qquad (x,y \in \M)$$
implies
\begin{equation}\label{zero1}
\|\,|\ip{f(x)}{f_n(y)}|^2-|\ip{x}{y}|^2\,\|\to 0\ (\mbox{as}\
n\to\infty) \qquad (x,y \in \M).
\end{equation}
For $x=y$ and $n=m$, (\ref{e1}) yields
\begin{equation}\label{e2}
\|\ip{f_n(x)}{f_n(x)}-\ip{x}{x}\|\leq c^{2n}\varphi(c^{-n}x,c^{-n}x)
\qquad (x \in {\mathcal M}).
\end{equation}
Then we have
$$
\Vert \ip{f_n(x)}{f_n(x)}\Vert-\Vert\ip{x}{x}\Vert
\leq\|\ip{f_n(x)}{f_n(x)}-\ip{x}{x}\|\leq
c^{2n}\varphi(c^{-n}x,c^{-n}x)
$$
for all $x\in {\mathcal M}$, whence
\begin{equation*}
\|f_n(x)\|^2\leq\|x\|^2+ c^{2n}\varphi(c^{-n}x,c^{-n}x) \qquad (x\in
{\mathcal M}).
\end{equation*}
Let us fix $x \in \M$. By (\ref{phi}-b), the sequence $(f_n(x))$
in $\N$ is norm-bounded and therefore, due to {\sf [H]}, there
exists a subsequence of $(f_n(x))$ (for simplicity we shall assume
that $(f_n(x))$ has such a property) and $F(x) \in \N$ such that
\begin{equation*}
\|\ip{f_n(x)}{v}- \ip{F(x)}{v}\|\to 0\ (\mbox{as}\ n\to\infty)\qquad
(v\in {\mathcal N}).
\end{equation*}
By the continuity of multiplication and the continuity of involution $\ast$, this yields
\begin{equation}\label{first}
\|\, \abs{\ip{f_n(x)}{v}}^2- \abs{\ip{F(x)}{v}}^2\, \|\to 0\
(\mbox{as}\ n\to\infty)\qquad (v\in {\mathcal N}),
\end{equation}
as well as
\begin{equation}\label{second}
\|\, \abs{\ip{v}{f_n(x)}}^2- \abs{\ip{v}{F(x)}}^2\,\|\to 0\
(\mbox{as}\ n\to\infty)\qquad (v\in {\mathcal N}).
\end{equation}
Thus we have defined the mapping $F : \M \to \N$ such that
(\ref{first}) and (\ref{second}) are true for each $x \in \M$. In
particular, (\ref{first}) implies
$$\|\abs{\ip{f_n(x)}{f(y)}}^2-
\abs{\ip{F(x)}{f(y)}}^2\|\to 0\ (\mbox{as}\ n\to\infty) \qquad (x,y
\in \M).$$ Hence, because of (\ref{zero}),
\begin{equation}\label{Ff}
\abs{\ip{F(x)}{f(y)}}^2=\abs{\ip{x}{y}}^2 \qquad (x,y \in \M).
\end{equation}
Inserting $c^{-n}y$ instead of $y$, we obtain
\begin{equation}\label{fn}
\abs{\ip{F(x)}{f_n(y)}}^2=\abs{\ip{x}{y}}^2 \qquad (x,y \in \M).
\end{equation}
Letting $n \to \infty$ this yields
$$\abs{\ip{F(x)}{F(y)}}^2=\abs{\ip{x}{y}}^2 \qquad (x,y \in \M),$$
and finally
\begin{equation}\label{FF}
\abs{\ip{F(x)}{F(y)}}=\abs{\ip{x}{y}} \qquad (x,y \in \M).
\end{equation}
In particular,
$$\ip{F(x)}{F(x)}=\ip{x}{x} \qquad (x \in \M).$$
Note that (\ref{Ff}) implies
\begin{equation}\label{Ffx}
\abs{\ip{F(x)}{f(x)}}^2=\langle x, x \rangle^2 \qquad (x \in \M).
\end{equation}%
From (\ref{second}) we get
$$\|\,\abs{\ip{f(x)}{f_n(x)}}^2-
\abs{\ip{f(x)}{F(x)}}^2\,\|\to 0\ (\mbox{as}\ n\to\infty) \qquad (x
\in \M),$$ which implies, because of (\ref{zero1}),
\begin{equation}\label{fFx}
\abs{\ip{f(x)}{F(x)}}^2=\langle x, x \rangle^2 \qquad (x \in \M).
\end{equation}
Comparing (\ref{Ffx}) and (\ref{fFx}) we conclude
$$\abs{\ip{F(x)}{f(x)}}^2 = \abs{\ip{f(x)}{F(x)}}^2 \qquad (x \in \M).$$
Hence, $\ip{F(x)}{f(x)}$ is a normal element in $\B$ for every $x
\in \M.$ Let us fix an arbitrary $x \in \M.$ Let $\B(x)$ be the von
Neumann algebra generated by the set $\{ \langle F(x), f(x) \rangle,
\langle f(x), F(x) \rangle, id_{\Ha}\}.$ Then $\B(x)$ is abelian
(cf.\ e.g.~\cite[p.~117]{mur}) and $\B(x) \subseteq \B.$

Using the polar decomposition we can find a partial isometry $s(x)
\in \B(x)$ such that
$$s(x)\abs{\ip{F(x)}{f(x)}} = \ip{F(x)}{f(x)} \quad \textup{ and } \quad s(x)^*\langle F(x), f(x) \rangle = \abs{\langle F(x), f(x) \rangle}.$$
Since $\abs{\langle F(x), f(x) \rangle}=\langle x, x \rangle,$ this
can be written as
$$s(x)\ip{x}{x} = \ip{F(x)}{f(x)} \quad \textup{ and } \quad s(x)^*\langle F(x), f(x) \rangle = \langle x, x \rangle.$$
In particular,
$$s(x)^*s(x)\langle x, x \rangle=s(x)^*\langle F(x), f(x) \rangle = \langle x, x \rangle.$$
Since $\B(x)$ is abelian and $\langle x, x \rangle =
\abs{\ip{F(x)}{f(x)}} \in \B(x),$ we conclude that all elements in
$\B(x)$ commute with $\langle x, x \rangle.$ If we define
$p(x)=s(x)^*s(x),$ then $p(x)$ is a projection in $\B(x)$ such that
$$p(x)\langle x, x \rangle=\langle x, x \rangle p(x) = \langle x, x \rangle.$$
Since $\langle F(x), F(x) \rangle = \langle x, x \rangle,$ this
implies
\begin{eqnarray}
\langle F(x)p(x)-F(x), F(x)p(x)-F(x) \rangle &=& p(x)\langle F(x), F(x) \rangle p(x)\nonumber \\
&& \mbox{}- p(x)\langle F(x), F(x) \rangle  \nonumber \\
&& \mbox{}- \langle F(x), F(x) \rangle p(x)\nonumber\\
&&\mbox{}+ \langle F(x), F(x) \rangle \nonumber \\
&=& p(x)\langle x, x \rangle p(x) - p(x)\langle x, x \rangle  \nonumber \\
&& \mbox{}- \langle x, x \rangle p(x)+ \langle x, x \rangle = 0
\nonumber.
\end{eqnarray}
Thus
\begin{equation}\label{Fp}
F(x)p(x)=F(x) \qquad (x \in \M).
\end{equation}
Let us define
$$I(x)=F(x)s(x) \in \N.$$
Then
$$\langle I(x), f(x) \rangle = s(x)^* \langle F(x), f(x) \rangle = \langle x, x \rangle,$$
whence by taking the adjoint,
\begin{equation}\label{fit}
\langle f(x), I(x) \rangle = \langle x, x \rangle.
\end{equation}

We have defined a mapping $I\colon{\mathcal M}\to{\mathcal N}$. We
will show that it satisfies the desired properties.

First we have, for all $x \in \M$,
\begin{eqnarray}\label{fitt}
\ip{I(x)}{I(x)}&=&s(x)^*\ip{F(x)}{F(x)}s(x) \nonumber \\
&=&s(x)^*\ip{x}{x}s(x)= s(x)^*s(x)\ip{x}{x}\\
&=&p(x)\ip{x}{x}=\ip{x}{x}.\nonumber
\end{eqnarray}
This implies
\begin{eqnarray*}
\langle f(x)-I(x), f(x)-I(x) \rangle &=& \langle f(x), f(x) \rangle - \langle f(x), I(x) \rangle\\
&& \mbox{}- \langle I(x), f(x) \rangle + \langle I(x), I(x) \rangle \\
&=& \langle f(x), f(x) \rangle - \langle x, x \rangle - \langle x, x \rangle + \langle x, x \rangle \\
&=& \langle f(x), f(x) \rangle - \langle x, x \rangle,
\end{eqnarray*}
which yields
$$\Vert f(x) - I(x) \Vert^2 = \Vert \langle f(x), f(x) \rangle - \langle x, x \rangle \Vert \leq \varphi(x,x) \qquad (x \in \M)$$
and finally
$$\Vert f(x) - I(x) \Vert \leq \sqrt{\varphi(x,x)} \qquad (x \in \M).$$

If $\langle x, y \rangle =0$, then (\ref{FF}) implies $\langle F(x),
F(y) \rangle =0$. This yields $\langle I(x), I(y) \rangle = s(x)^*
\langle F(x), F(y) \rangle s(y) =0$.

Since $s(x) \in \B(x)$ and $\B(x)$ is an abelian von Neumann
algebra, $s(x)s(x)^*=s(x)^*s(x)=p(x)$, so
$$
F(x)=F(x)p(x)=F(x)s(x)s(x)^*=I(x)s(x)^*. $$
If $\langle I(x),
I(y)\rangle=0$, then $\langle F(x), F(y)\rangle =s(x)\langle
I(x),I(y)\rangle s(y)^*=0$ and (\ref{FF}) implies $\langle
x,y\rangle=0$. Hence, $I$ preserves orthogonality in both
directions.

If we define $h(x):=f(x)-I(x)$ for all $x \in \M$, then, by
(\ref{fit}) and (\ref{fitt}),
$$\langle h(x), I(x) \rangle = \langle f(x), I(x) \rangle - \langle I(x), I(x) \rangle
= \langle x, x \rangle - \langle x, x \rangle = 0.$$

Now assume that $\B$ is abelian. According to (\ref{FF}) we have
\begin{eqnarray}
\abs{\langle I(x), I(y) \rangle}^2 &=& s(y)^* \langle F(y), F(x) \rangle s(x)s(x)^* \langle F(x), F(y) \rangle s(y) \nonumber \\
&=& p(y) \langle F(y), F(x) \rangle p(x) \langle F(x), F(y) \rangle \nonumber \\
&=& \langle F(y)p(y), F(x)p(x) \rangle \langle F(x), F(y) \rangle \nonumber \\
&=& \abs{\langle F(x), F(y) \rangle}^2 = \abs{\langle x, y
\rangle}^2 \nonumber \quad (x,y \in \M),
\end{eqnarray}
so $I$ is a solution of {\rm (W)}.
\end{proof}
\begin{remark}
The conclusion (i) of Theorem \ref{t1} implies $\langle x, x
\rangle = \langle I(x), I(x) \rangle \in \B$. Since $\langle x, x
\rangle \in \B$ for all $x \in \mathcal M$, we have $\langle x,y\rangle =
1/4\sum_{i=0}^3 i^k\langle x+i^ky, x+i^ky\rangle \in \B$ for all $x,y \in \mathcal M$. If
$\mathcal M$ is a full Hilbert $\A$-module (that is, the ideal
$\langle \mathcal M, \mathcal M \rangle$ generated by all products
$\ip{x}{y}$, $x,y\in\M$, is dense in $\A$), this yields $\A
\subseteq \B.$
\end{remark}


\section{Applications and notes}

The \textit{orthogonality preserving property} of the mapping $I$
in Theorem \ref{t1} leads to a question on how this property is
related to the Wigner equation.  Such a relation was shown in the
realm of Hilbert spaces and under some linearity-type assumptions
(cf.~e.g.~\cite[Corollary 2.4]{jch2006}). Linear orthogonality
preserving mappings have also been studied in normed spaces, with
the Birkhoff-James or  the semi-inner product orthogonalities
(cf.~\cite{kr, koldobsky, b-t, chm-bjma}).
 It seems that the
investigation on such a class of mappings between inner product
modules would be of independent interest; some recent results in
that direction can be found in \cite{i-t}.

The following result is a consequence of Theorem \ref{t1} and
Proposition \ref{prop}.
\begin{corollary}
Let $\A$ and $\B$ be finite-dimensional $C^*$-algebras acting on a
Hil\-bert space $\Ha$ and let $\B$ be abelian and containing the
identity operator $id_{\Ha}.$ Let ${\M}$ be an inner product
$\A$-module and let $\N$ be a Hilbert $\B$-module. Then, for each
mapping $f\colon{\mathcal M} \to {\mathcal N}$ satisfying {\rm
(W$_{\varphi}$)}, with $\varphi$ satisfying (\ref{phi}), there
exists a solution $I\colon{\mathcal M} \to {\mathcal N}$ of {\rm
(W)} such that
$$
\|f(x)-I(x)\|\leq\sqrt{\varphi(x,x)} \qquad (x\in {\mathcal M}).
$$
\end{corollary}

In particular, for $\A=\B=\mathbb{C}$ we obtain the stability of the
Wigner equation between an inner product space $\M$ and a Hilbert
space $\N$.

One can consider the control mapping
$$\varphi(x,y):=\varepsilon
\|x\|^p\|y\|^q,
$$
where $\varepsilon>0$ and $p, q$ are fixed real numbers such that
either $p,q>1$ or $p, q< 1$ (we assume $\|0\|^0=1$ and
$\|0\|^p=\infty$ for $p < 0$). It is easy to check that the
conditions (\ref{phi}) are satisfied with $c=2$ (if $p,q> 1$) or
$c=\frac{1}{2}$ (if $p, q< 1$). Then the results from the previous
section yield the following result.


\begin{corollary}
Let $\A$ be a $C^*$-algebra and $\B$ be a von Neumann algebra that
both act on a Hilbert space $\Ha$. Let ${\M}$ be an inner product
$\A$-module and let $\N$ be an inner product $\B$-module
satisfying {\sf [H]}. Let either $p, q> 1$ or $p, q< 1$ and
$\varepsilon
>0$. Then, for each mapping $f\colon{\mathcal M} \to {\mathcal N}$
satisfying
$$
\|\,|\ip{f(x)}{f(y)}|-|\ip{x}{y}|\,\|\leq\varepsilon \|x\|^p\|y\|^q
\qquad (x,y\in {\mathcal M}),
$$
there exists $I \colon{\mathcal M} \to {\mathcal N}$ with the
following properties:
\begin{itemize}
\item[(i)] $\langle I(x), I(x) \rangle = \langle x, x \rangle
\qquad (x \in \M)$, \item[(ii)] $I$ preserves orthogonality in both
directions,
\item[(iii)] $\|f(x)-I(x)\|\leq\sqrt{\varepsilon}\,\|x\|^{(p+q)/2}
\qquad (x\in {\mathcal M})$.
\end{itemize}
Moreover, if $\B$ is abelian, then $I$ can be chosen as a solution
of {\rm (W)}.
\end{corollary}


\begin{remark}
The case when the control mapping is of the form  $\varphi(x,y) =
\varepsilon  \|x\|\,\|y\|$ is still unsolved even in the framework
of Hilbert spaces (cf.~\cite{jch}).
\end{remark}


As a consequence of Theorem \ref{t1}, we have obtained the already
known stability results in the category of Hilbert spaces.
However, in Hilbert spaces it has been proved
(cf.~\cite{jch-survey}) that the mapping $I$ appearing in the
assertion of Theorem \ref{t1} is unique up to phase-equivalence.
The proof used, in particular, the Wigner's theorem and the fact
that the equality in Cauchy-Schwarz inequality yields linear
dependence of vectors. However, in the setting of Hilbert
$C^*$-modules Wigner's theorem has not been established generally
(it has been investigated only for some specific classes of
$C^*$-algebras, e.g.~\cite{b-g1, b-g2}). The uniqueness of $I$ in
the assertion of Theorem \ref{t1} remains an open problem.

\subsection*{Acknowledgments} The authors thank the anonymous
referee for his/her remarks; in particular for paying their
attention to the paper \cite{frank}.

The second author was in part supported by the Ministry of
Science, Education and Sports of the Republic of Croatia (project
No. 037-0372784-2757).


\end{document}